\documentclass[12pt]{amsart}

\usepackage{amsthm}
\usepackage{amssymb}
\usepackage{amsmath}
\usepackage{graphicx}
\usepackage[utf8]{inputenc}

\usepackage{color}
\usepackage[colorlinks=true,linkcolor=blue,citecolor=blue]{hyperref}
\usepackage{enumitem}

\usepackage{todonotes}

\usepackage{mathrsfs}

\newtheorem{theorem}{Theorem}[section]
\newtheorem*{theorem*}{Theorem}
\newtheorem{proposition}[theorem]{Proposition}
\newtheorem{corollary}[theorem]{Corollary}
\newtheorem{lemma}[theorem]{Lemma}
\theoremstyle{definition}
\newtheorem{definition}[theorem]{Definition}

\newtheorem{example}[theorem]{Example}

\theoremstyle{remark}

\usepackage[normalem]{ulem}

\newtheorem{remark}[theorem]{Remark}

\numberwithin{equation}{section}

\allowdisplaybreaks

\begin{document}

\title[Algebrability of the set of hypercyclic vectors]{Algebrability of the set of hypercyclic vectors for backward shift operators}

\thanks{The authors are grateful to Prof. Manuel Maestre for a very helpful discussion.\\
\indent The authors were supported by the Fonds de la Recherche Scientifique - FNRS, grant no. PDR T.0164.16}



\author[J. Falc\'{o}]{Javier Falc\'{o}}
\address[Javier Falc\'{o}]{D\'epartement de Math\'ematique, Institut Complexys,
Universit\'e de Mons, 20 Place du Parc, 7000 Mons, Belgium} 
\curraddr{Departamento de An\'alisis Matem\'atico, Universidad de Valencia, Doctor Moliner 50, 46100 Burjasot (Valencia),
Spain.}
\email{Francisco.J.Falco@uv.es}
\author[K.-G. Grosse-Erdmann]{Karl-G. Grosse-Erdmann}
\address[Karl-G. Grosse-Erdmann]{D\'epartement de Math\'ematique, Institut Complexys,
Universit\'e de Mons, 20 Place du Parc, 7000 Mons, Belgium} \email{kg.grosse-erdmann@umons.ac.be}
%

%\date{\today}

\begin{abstract}
	
We study the existence of algebras of hypercyclic vectors for weight\-ed backward shifts on Fr\'echet sequence spaces that are algebras when endowed with coordinatewise multiplication or with the Cauchy product. As a particular case we obtain that the sets of hypercyclic vectors for Rolewicz's and MacLane's operators are algebrable.
	
\end{abstract}

\maketitle

One of the aims of linear dynamics is to study and understand the structure and the properties of the set of hypercyclic vectors of an operator $T$ on a Fr\'echet space $X$,
$$ 
HC(T) = \big\{x \in X : \{x, T x, T^{2} x, \ldots \} \text{ is dense in }X\big\}.
$$ 
It is well known that the set $HC(T)$ is either empty or contains a dense linear subspace (but the origin), see \cite[Theorem 2.55]{GrPe11}. However, when the underlying vector space $X$ possesses a richer structure it is natural to ask whether the set of hypercyclic vectors for a given hypercyclic operator $T$ on $X$ also has a richer structure in the same spirit. For instance, the Fr\'echet space $H(\mathbb C)$ of entire functions can be naturally endowed with the multiplicative structure given by the pointwise multiplication of functions, which leads to an algebraic structure of the space. In fact, the space $H(\mathbb C)$ endowed with pointwise multiplication is a Fr\'echet algebra. 

When we have a hypercyclic operator $T$ on a Fr\'echet algebra $X$ it is natural to ask if $HC(T)$ contains a non-trivial subalgebra  of $X$ (except  zero). When such a subalgebra exists it is called a \textit{hypercyclic algebra} for $T$. If a hypercyclic algebra is infinitely but not finitely generated then $HC(T)$ is said to be \textit{algebrable}; see the monograph \cite{ABPS16} for this and related notions.

Aron et al. \cite{Aron07} showed that not every hypercyclic operator on a Fr\'echet algebra contains a hypercyclic algebra. Indeed, no translation operator $\tau_{a}$ on $H(\mathbb C)$, 
$$
\tau_{a} (f)(z) = f (z + a), \quad f \in  H(\mathbb C),\ z \in \mathbb C,\  a\neq 0,
$$
can support a hypercyclic algebra; these operators are also called Birkhoff's operators in linear dynamics, see \cite{GrPe11}. In contrast, in the same paper it is shown that there exists a function $f\in H(\mathbb C)$ such that all the powers of $f$ are in $HC(D)$, where
\[
(Df)(z)=f'(z), \quad f \in  H(\mathbb C),\ z \in \mathbb C,
\]
defines the complex differentiation operator, also called MacLane's operator. This result gave hope for the existence of hypercyclic algebras. Shortly afterwards,  Shkarin \cite{Sh10} and Bayart and Matheron \cite[Theorem 8.26]{EtMAt} showed independently that $D$ admits a hypercyclic algebra, thereby providing the first known example of an operator that admits a hypercyclic algebra. While the proof of Shkarin is purely constructive the approach used by Bayart and Matheron makes use of a Baire argument. Using the ideas of Bayart and Matheron, B\`{e}s, Conejero and Papathanasiou \cite{BesConPap} extended the result to convolution operators $P(D)$ induced by non-constant polynomials $P$ that vanish at zero. In \cite{BesConPap2}, using a different method, they obtained a further extension to convolution operators $\Phi(D)$ for various entire functions $\Phi$ of exponential type, including functions that do not vanish at zero.

Here we improve the result of Shkarin, Bayart and Matheron in two ways. First, we consider general weighted backward shift operators on Fr\'echet sequence algebras, where the multiplicative structure can be given either by coordinatewise multiplication or by the Cauchy product (that is, the discrete convolution). The importance of this scenario is that many operators can be seen as weighted backward shifts on a Fr\'echet sequence space. Our results therefore cover the multiples $\lambda B$, $|\lambda|>1$, of the backward shift operator on $\ell^p$-spaces and on $c_0$, also called Rolewicz's operators, as well as MacLane's operator $D$ on $H(\mathbb{C})$. And secondly, we obtain in each case that the set of hypercyclic vectors is algebrable. For MacLane's operator this gives a positive answer to a question posed by Aron \cite[p. 217]{EtMAt}.

Aron's problem has already been solved recently by B\`es and Papathanasiou \cite{BesPap}, using Baire's theorem; they even obtain dense algebrability. Our solution was obtained independently, and its proof is purely constructive. In fact, the two papers obtain far-reaching generalizations in different directions: B\`es and Papathanasiou regard $D$ as a particular convolution operator, we look at it as a particular weighted backward shift.

In the first section we establish some notation and terminology that we will use during the rest of the paper. In Sections \ref{pointwise} and \ref{cauchy} we give sufficient conditions for a weighted backward shift on a Fr\'echet sequence algebra to contain a hypercyclic algebra under coordinatewise multiplication and under Cauchy products, respectively.

\section{Notation and terminology}\label{sec1}

We consider a complex $m$-convex Fr\'echet algebra, that is an algebra $X$ over the complex numbers that at the same time is a (locally convex) Fr\'echet space whose topology is induced by an increasing sequence $(\|\cdot\|_q)_{q\geq 1}$ of seminorms 
that are submultiplicative, i.e.,
\begin{equation}\label{eq:subm}
\|xy\|_q\leq \|x\|_q \|y\|_q
\end{equation}
for all $x,y\in X$, $q\geq 1$. For brevity we will call $X=(X, (\|\cdot\|_q)_q)$ simply a \textit{Fr\'echet algebra}; see \cite{Fra}.

The space of all complex sequences is denoted, as usual, by 
$$
\omega=\{x=(x_{n})_{n\geq 0}:x_{n}\in\mathbb C, n\in\mathbb N_0\}.
$$
We endow $\omega$ with the product topology, that is, the topology of coordinatewise convergence. 

A \textit{sequence space} is a subspace of $\omega$. As for a multiplicative structure one may endow $\omega$ either with the
coordinatewise product of sequences, see Section \ref{pointwise}, or with the Cauchy product of sequences, see Section \ref{cauchy}. A \textit{sequence algebra} is a subalgebra of $\omega$ in either of the two senses. 

The sequence $e_n$, $n\geq 0$, is defined as $e_n=(0,\ldots,0,1,0,\ldots)$ with 1 at index $n$. Furthermore, we write $e=(1,1,1\ldots)$.

We denote by 
$$
\varphi=\Big\{\sum_{n=0}^{N}x_{n}e_{n}:x_0,\ldots,x_N\in\mathbb{C}, N\in\mathbb{N}_0\Big\}
$$ 
the set of all \textit{finite sequences}.

When a sequence space, respectively sequence algebra, $X$ carries the additional structure of a Fr\'echet space, resp. Fr\'echet algebra, such that the canonical embedding into $\omega$ is continuous we speak of a \textit{Fr\'echet sequence space}, resp. \textit{Fr\'echet sequence algebra}. 

A weighted backward shift on $\omega$ is an operator $B_{w}$ given by 
$$
B_{w}(x_{0}, x_{1}, x_{2}, \ldots) = (w_{1}x_{1}, w_{2}x_{2}, w_{3}x_{3}, \ldots),\quad x\in\omega,
$$
where $w = (w_{n})_{n\geq 0}$ is a sequence of non-zero complex numbers, called a \textit{weight sequence}. The unweighted shift is denoted by $B=B_e$. The forward shift associated to a weight $w$ is the operator given by 
$$
F_{w}(x_{0}, x_{1}, x_{2}, \ldots) = (0, w_{1}x_{0}, w_{2}x_{1}, w_{3}x_{2}, \ldots),\quad x\in\omega.
$$
Naturally we have that  $B_{w}F_{w^{-1}}=I$, where $I$ is the identity map on $\omega$ and $w^{-1}=(w_n^{-1})_n$. Since the element $w_{0}$ is not relevant for the definition of the operators $B_{w}$ and $F_{w}$ we will assume that $w_{0}=1$ for any weight $w$.

Throughout the paper we will write, for a given weight $w=(w_n)_n$,
\[
v_{n}=\prod_{k=0}^{n}w_{k}, \quad n\geq 0.
\]

Note that the closed graph theorem implies that as soon as $B_w$ or $F_w$ maps a Fr\'echet sequence space $X$ into itself then it defines a (continuous, linear) operator on $X$.

Apart from the notion of hypercyclicity we will need the stronger property of mixing. An operator $T$ on a separable Fr\'echet space $X$ is called \textit{mixing} if, for any non-empty open subsets $U,V$ of $X$, the set $\{n\geq 0: T^n(U)\cap V\neq \varnothing\}$ is co-finite. 

For monographs on linear dynamics we refer to \cite{EtMAt} and \cite{GrPe11}.

%%%%%%%%%%%%%%%%%%%%%%%%%%%%%%%%%%%%%%
%%%%%%%%%%%%%%%%%%%%%%%%%%%%%%%%%%%%%%
%%%%%%%%%%%%%%%%%%%%%%%%%%%%%%%%%%%%%%

\section{Fr\'echet sequence algebras under coordinatewise multiplication}
\label{pointwise}

In this section we study algebrability of the set of hypercyclic vectors by considering dynamical systems where the underlying space $X$ is a Fr\'echet sequence algebra and the multiplicative structure is the coordinatewise multiplication of sequences. So given two sequences  $x=(x_n)_{n}$ and  $y=(y_n)_{n}$ in $X$ we define  $xy = (x_ny_n)_{n}$. We will assume that the weighted backward shift $B_w$ is an operator on~$X$.

To start, it will be useful to consider the following variant of the well-known characterization of hypercyclicity of weighted backward shifts, see \cite[Theorem 4.8]{GrPe11}.

\begin{proposition} 
\label{prop:increasingsequence}
Let $X$ be a Fr\'echet sequence space in which $(e_{n})_{n}$ is a basis. Suppose that the weighted backward shift $B_{w}$ is an operator on $X$. Then $B_{w}$ is hypercyclic if and only if there exists an increasing sequence $(p_{k})_{k}$ of natural numbers such that
\begin{equation}\label{eq1}
\text{for each $n\geq 0$,}\quad v_{p_{k}+n}^{-1}e_{p_{k}+n}\to 0
\end{equation}
in $X$ as $k\to \infty$.
\end{proposition}

\begin{proof}
We have that $B_w$ is hypercyclic if and only if there exists an increasing sequence $(m_{k})_{k}$ of natural numbers such that
\begin{equation}\label{eq2}
v_{m_{k}}^{-1}e_{m_{k}}\to 0
\end{equation}
in $X$ as $k\to \infty$, see \cite[Theorem 4.8]{GrPe11}. Thus condition \eqref{eq1} is sufficient. 

For the necessity, suppose that \eqref{eq2} holds. It follows from continuity of $B_w$ and the fact that $B_we_k= w_ke_{k-1}$, $k\geq 1$, that
\[
v_{m_{k}-n}^{-1}e_{m_{k}-n}= B_w^n v_{m_{k}}^{-1}e_{m_{k}} \to 0
\]
as $k\to\infty$, for any $n\geq 0$. By \cite[Lemma 4.2]{GrPe11} there is an increasing sequence $(p_k)_k$ of natural numbers such that
\[
v_{p_{k}+n}^{-1}e_{p_{k}+n} \to 0
\]  
as $k\to\infty$, for any $n\geq 0$, which had to be shown. 
\end{proof}

\begin{definition}\label{def:propertyCS}
Let $(X, (\|\cdot\|_q)_q)$ be a Fr\'echet sequence space that contains the finite sequences. We say that $(e_n)_n$ has \textit{Property A} if, for any $r\geq 1$, there is some $q\geq 1$ and some $C>0$ such that, for all $n\geq 0$,
\begin{equation}\label{eq3}
\|e_n\|_r^2\leq C \|e_n\|_q.
\end{equation}
\end{definition}

This is less of a restriction than it might at first appear.

\begin{example}	\label{propA}
(a) If $(e_n)_n$ is bounded in the space $X$ then it has Property A; simply consider $q=r$. In particular, the classical sequence spaces $\ell^p$, $1\leq p<\infty$, and $c_0$ are Banach sequence algebras under their usual norms and coordinatewise multiplication (\eqref{eq:subm} is easily verified) for which their bases $(e_n)_n$ have Property A. 

(b) The space $H(\mathbb{C})$ of entire functions can be considered as a sequence space via Taylor coefficients at 0. Its natural topology of uniform convergence on compact sets can be induced by 
the seminorms
\[
\|(a_n)_{n\geq 0}\|_q = \sum_{n=0}^\infty |a_n|q^n,\quad q\geq 1.
\]
This turns $H(\mathbb{C})$ into a Fr\'echet sequence algebra under coordinatewise multiplication of the sequences (\eqref{eq:subm} is easily verified). Moreover, its basis $(e_n)_{n}$ has Property A since $\|e_n\|_r^2= \|e_n\|_{r^2}$ for $n\geq 0$.

(c) The product topology of the space $\omega$ of all sequences is generated by the increasing sequence of seminorms
$$
\|x\|_q = \sup_{0\leq n\leq q} \vert x_{n}\vert,\quad q\geq 1.
$$ 
Then $\omega$ is a Fr\'echet sequence algebra under coordinatewise multiplication whose basis $(e_n)_{n}$ has Property A.
\end{example}

We will need the following improvement of Property A.

\begin{lemma}\label{lem:propertyCS}
Let $(X, (\|\cdot\|_q)_q)$ be a Fr\'echet sequence space for which $(e_n)_n$ has Property A. Then, for any $m\geq 1$ and $r\geq 1$, there is some $q\geq 1$ and some $C>0$ such that, for all $n\geq 0$,
\begin{equation}\label{eq3b}
\|e_n\|_r^m\leq C \|e_n\|_q.
\end{equation}
\end{lemma}

\begin{proof} It follows in view of the definition of Property A that the result holds for $m=2^N$, $N\geq 0$. If $2^N\leq m< 2^{N+1}$, then the result follows from the fact that $\|e_n\|_r^m\leq \max\{\|e_n\|_r^{2^N}, \|e_n\|_r^{2^{N+1}}\}$. 
\end{proof}

As an application we obtain an improvement of \eqref{eq1} under Property A.

\begin{lemma}\label{coro:increasingsequenceroots}
Let $(X, (\|\cdot\|_q)_q)$ be a Fr\'echet sequence space for which $(e_n)_n$ has \textit{Property A}. Let $w=(w_n)_n$ be a weight. If $(p_{k})_{k}$ is an increasing sequence of natural numbers that satisfies \eqref{eq1} then, for any $m\geq 1$, $n\geq 0$,
\begin{equation*}
v_{p_{k}+n}^{-\frac{1}{m}}e_{p_{k}+n}\to 0
\end{equation*}
in $X$ as $k\to \infty$, were $v_{n}^{-\frac{1}{m}}$ is any $m$-th root of $v_{n}^{-1}$ in $\mathbb C$.
\end{lemma}

\begin{proof}
Let $m\geq 1$ and $r\geq 1$. Then there are $q\geq 1$ and $C>0$ such that \eqref{eq3b} holds for all $n\geq 0$. Hence we have, for any $n\geq 0$, 
\begin{align*}
\|v_{p_{k}+n}^{-\frac{1}{m}}e_{p_{k}+n}\|_r&=|v_{p_{k}+n}|^{-\frac{1}{m}}\|e_{p_{k}+n}\|_r\\
&=\big(|v_{p_{k}+n}|^{-1}\|e_{p_{k}+n}\|_r^m\big)^{\frac{1}{m}}\\
&\leq\big(|v_{p_{k}+n}|^{-1}\ C \|e_{p_{k}+n}\|_q\big)^{\frac{1}{m}}\quad(\text{by \eqref{eq3b}})\\
&=\big(C\|v_{p_{k}+n}^{-1}e_{p_{k}+n}\|_q\big)^{\frac{1}{m}}\to 0\quad(\text{by \eqref{eq1}})
\end{align*}
as $k\to \infty$. 
\end{proof}

Now we present our first result on the existence of algebras of hypercyclic vectors for weighted backward shifts on Fr\'echet sequence algebras.
 
\begin{theorem}\label{thrm:algebra1}
Let $(X, (\|\cdot\|_q)_q)$ be a Fr\'echet sequence algebra under coordinatewise multiplication in which $(e_{n})_{n}$ is a basis with Property A. Let $B_{w}$ be a hypercyclic weighted backward shift on $X$. If there exists an increasing sequence $(p_k)_k$ of natural numbers  satisfying \eqref{eq1} such that
\[
\text{for any $n\geq 0$,}\quad\prod_{\nu=0}^{p_k+n}w_{\nu}^{-1}\to 0 \quad\text{as $k\to\infty$},
\]
then there exists a point $x\in HC(B_{w})$ such that the algebra generated by $x$, except zero, is contained in $HC(B_{w})$.
\end{theorem}

\begin{proof}
To simplify our notation we will denote by $T$ the weighted backward shift $B_{w}$ on $X$.

We will associate each number $r\in\mathbb{N}$ with the $r$-th element of a fixed order in the set $\mathbb{N}\times \mathbb{N}$, and we simply write $r=(m,l)$.

For each natural number $m$ let us fix an $m$-th root of $w_{n}$, $n\geq 0$, which we denote by $w_{n}^{\frac{1}{m}}$; the $j$-th power of the latter number is denoted by $w_{n}^{\frac{j}{m}}$. Note that one has to distinguish, for example, $w_n^\frac{1}{2}$ from $w_n^\frac{2}{4}$.
 
Since $(e_{n})_{n}$ is a basis of $X$, $\varphi$ is dense in $X$. Let $(y^{(l)})_{l\geq 1}\subset \varphi$ be a dense sequence of non-zero points  in $X$ such that for each $l_{0}\in\mathbb N$ the element $y^{(l_{0})}$ appears infinitely many times in the sequence  $(y^{(l)})_{l}$. Let $s_{l}$ be the largest index of the non-zero coordinates of $y^{(l)}$. As before, for any $m\geq 1$, we fix an $m$-th root of $y^{(l)}_{n}$, $l\geq 1$, $n\geq 0$, written $(y^{(l)}_{n})^{\frac{1}{m}}$, and we denote the $j$-th power of that number by $(y^{(l)}_{n})^{\frac{j}{m}}$. 

Let $a,j,m,l\geq 1$.  We will in the sequel denote by
\[
( S^{a}y^{(l)})^{\frac{j}{m}}
\]
the $j$-th power of the point $F^{a}_{w^{-\frac{1}{m}}}(y^{(l)})^{\frac{1}{m}}$, where $w^{-\frac{1}{m}}=(1/w_{0}^{\frac{1}{m}}, 1/w_{1}^{\frac{1}{m}},\ldots)$ and $(y^{(l)})^{\frac{1}{m}}= ((y^{(1)}_{n})^{\frac{1}{m}}, (y^{(2)}_{n})^{\frac{1}{m}},\ldots)$. In other words,
\begin{equation}\label{eq:0}
( S^{a}y^{(l)})^{\frac{j}{m}} = \sum_{n=0}^{s_{l}} \frac{1}{w_{n+1}^{\frac{j}{m}}\cdots w_{n+a}^{\frac{j}{m}}} (y_n^{(l)})^{\frac{j}{m}}e_{n+a};
\end{equation}
in particular, $( S^{a}y^{(l)})^{\frac{m}{m}}=F^{a}_{w^{-1}}y^{(l)}$, so that 
\begin{equation}\label{eq:00}
T^a( S^{a}y^{(l)})^{\frac{m}{m}}=y^{(l)}. 
\end{equation}

We will now construct an increasing sequence of natural numbers $(a_{r})_{r\geq 1}$ with  $a_{r}\in \{p_k: k\geq 1\}$ and such that, if $r=(m,l)\in \mathbb N$, then
\begin{enumerate}[label=\textbf{A.\arabic*}]
\item \label{condition1} $\|	(S^{a_{r}} y^{(l)})^{\frac{1}{m}}\|_r< 2^{-r}$,
\end{enumerate}
and if $r\geq 2$ then
\begin{enumerate}[label=\textbf{A.\arabic*},resume]
\item \label{condition2} $\| T^{a_{t}}( S^{a_{r}}y^{(l)})^{\frac{\nu}{m}}\|_r< 2^{-r}$ for $1\leq t<r$ and $1\leq \nu \leq d_{r}$,
\item \label{condition3} $a_{r}-a_{r-1}> s_{\widetilde{l}}$ with $r-1=(\widetilde{m},\widetilde{l})$,
\end{enumerate}
where we set $d_{r}=\max_{(\widetilde{m},\widetilde{l})< r}\widetilde{m}$.

Let $(p_{k})_{k\geq 1}$ be an increasing sequence of natural numbers such that \eqref{eq1} holds (which exists by Proposition \ref{prop:increasingsequence}). By Lemma \ref{coro:increasingsequenceroots} we have that, for all $m\geq 1$, $n\geq 0$, 
\begin{equation}\label{eq:0b}
\frac{1}{w_{n+1}^{\frac{1}{m}}\cdots w_{n+p_k}^{\frac{1}{m}}}e_{n+p_{k}}=\frac{w_{0}^{\frac{1}{m}}\cdots w_{n}^{\frac{1}{m}}}{w_{0}^{\frac{1}{m}}\cdots w_{n+p_k}^{\frac{1}{m}}}e_{n+p_{k}} \to 0
\end{equation}
as $k\to\infty$. In view of \eqref{eq:0}, there exists $a_{1}\in \{p_k: k\geq 1\}$ that satisfies condition \ref{condition1}.

Let us now assume that we have fixed  $a_{1},\ldots, a_{r-1}$ $(r\geq 2)$ satisfying conditions \ref{condition1},  \ref{condition2} and \ref{condition3}. Assume $r=(m,l)$. Since multiplication is continuous in $X$, \eqref{eq:0b} implies that, for any $m,j\geq 1$, $n\geq 0$, 
\[
\frac{1}{w_{n+1}^{\frac{j}{m}}\cdots w_{n+p_k}^{\frac{j}{m}}}e_{n+p_{k}} \to 0
\]
as $k\to\infty$. Again by \eqref{eq:0}, and by continuity of $T$ at $0$, there is then some $a_r\in \{p_k: k\geq 1\}$, $a_r>a_{r-1}$, such that \ref{condition1}, \ref{condition2} and \ref{condition3} hold, and the induction process is completed.

In order to produce a hypercyclic algebra, we define  
\begin{equation}\label{eq:point}
x=\sum_{m=1}^{\infty} \sum_{l=1}^{\infty}	( S^{a_{(m,l)}}y^{(l)})^{\frac{1}{m}}=\sum_{r=1}^{\infty}(S^{a_{r}}y^{(l)})^{\frac{1}{m}}.
\end{equation}
By property \ref{condition1}, the series \eqref{eq:point} is convergent in $X$, so that $x\in X$. 

We first show that, for any $j\geq 1$, the $j$-th power of the point $x$ is hypercyclic for $T$. Fix a natural number $l_{0}$. Let us consider the number $t=(j,l_{0})$. Then 
\begin{align}\label{eq:pointb}
\begin{split}
T^{a_{t}}x^{j}&=T^{a_{t}}\Big(\sum_{r=1}^{\infty} 	(S^{a_{r}}y^{(l)})^{\frac{1}{m}} \Big)^{j}\\
&=T^{a_{t}}\sum_{r=t}^{\infty} (S^{a_{r}}y^{(l)})^{\frac{j}{m}}   \ \ \ (\text{by property \ref{condition3}})\\
&=T^{a_{t}}( S^{a_{t}}y^{(l_{0})})^{\frac{j}{j}}+\sum_{r=t+1}^{\infty} T^{a_{t}}(S^{a_{r}}y^{(l)})^{\frac{j}{m}} \\
&=y^{(l_{0})} +\sum_{r=t+1}^{\infty} T^{a_{t}}(S^{a_{r}}y^{(l)})^{\frac{j}{m}}.\ \ \ (\text{by } \eqref{eq:00})
\end{split}
\end{align}
Note that since $t=(j,l_{0})$ then $j\leq d_{r}$ for all $r>t$. Therefore, by property \ref{condition2},
\begin{equation}
\label{equation:2minusk}
\| T^{a_{t}}x^j - y^{(l_{0})}\|_t\leq \sum_{r=t+1}^{\infty} \| T^{a_{t}}(S^{a_{r}}y^{(l)})^{\frac{j}{m}}\|_r < 2^{-t}.
\end{equation}
Since the sequence $(y^{(l)})_{l}$ is dense in $X$ we have that $x^{j}\in HC(T)$ for any $j\geq 1$. 

To conclude, we show that any point $z\in X$ of the form
\[
z=\sum_{\nu=j}^{N}c_{\nu}x^{\nu}
\]
with $j\geq 1$ and $c_{j},\ldots,c_{N}\in \mathbb C$, $c_{j}\ne 0$, is hypercyclic for $T$. Since non-zero multiples of hypercyclic vectors are hypercyclic, we may assume that $c_{j}=1$, whence
\begin{equation}\label{eq:1}
T^{a_t}z=T^{a_t}x^j+\sum_{\nu=j+1}^{N}c_{\nu}T^{a_t}x^{\nu},\quad t\geq 1.
\end{equation} 

Let us fix a natural number $l_{0}$. Since the element $y^{(l_{0})}$ is repeated infinitely many times in the sequence $(y^{(l)})_{l}$ there exists an increasing sequence of natural numbers $(l_{i})_{i}$ with $y^{(l_{i})}=y^{(l_{0})}$ for all $i\geq 1$. By \eqref{equation:2minusk} we have for each $t=(j,l_{i})\in\mathbb N$, $i\geq 1$,
\begin{equation}
\label{eqtozeroj}
\|T^{a_{t}}x^{j} - y^{(l_{0})}\|_t < 2^{-t}.
\end{equation}

Also, we have for any $\nu\geq 1$
\begin{equation}\label{eq:2}
T^{a_{t}}x^{\nu} =T^{a_{t}}( S^{a_{t}}y^{(l_{0})})^{\frac{\nu}{j}}+\sum_{r=t+1}^{\infty} T^{a_{t}}(S^{a_{r}}y^{(l)})^{\frac{\nu}{m}}, 
\end{equation}
see \eqref{eq:pointb}.

For the first term, we obtain from \eqref{eq:0} that
\[
T^{a_{t}}( S^{a_{t}}y^{(l_{0})})^{\frac{\nu}{j}} = \sum_{n=0}^{s_{l_0}} \frac{w_{n+1}\cdots w_{n+a_t}}{w_{n+1}^{\frac{\nu}{j}}\cdots w_{n+a_t}^{\frac{\nu}{j}}} (y_n^{(l_0)})^{\frac{\nu}{j}}e_n.
\]
By hypothesis we have that $|v_{n+p_k}|\to\infty$ as $k\to\infty$, for all $n\geq 0$. Since, by construction, $(a_t)_t$ is a subsequence of $(p_k)_k$ we have for any $n\geq 0$ and $\nu>j$,
\[
\Big|\frac{w_{n+1}\cdots w_{n+a_t}}{w_{n+1}^{\frac{\nu}{j}}\cdots w_{n+a_t}^{\frac{\nu}{j}}}\Big| = \frac{|v_n|^{\frac{\nu}{j}-1}}{|v_{n+a_t}|^{\frac{\nu}{j}-1}}\to 0
\]
as $(t=(j,l_{i}))_{i}$ goes to infinity. This implies that, for $\nu>j$,
\begin{equation}\label{eq:4}
T^{a_{t}}( S^{a_{t}}y^{(l_{0})})^{\frac{\nu}{j}}\to 0.
\end{equation}

For the second term in \eqref{eq:2}, it follows from property \ref{condition2} that whenever $d_t\geq \nu$ then
\begin{equation}\label{eq:5}
\Big\|\sum_{r=t+1}^{\infty} T^{a_{t}}(S^{a_{r}}y^{(l)})^{\frac{\nu}{m}} \Big\|_t \leq \sum_{r=t+1}^{\infty} \|T^{a_{t}}(S^{a_{r}}y^{(l)})^{\frac{\nu}{m}} \|_r < 2^{-t}.
\end{equation}

Therefore, by equations \eqref{eq:1}, \eqref{eqtozeroj}, \eqref{eq:2}, \eqref{eq:4} and \eqref{eq:5} we have that
\begin{align*}
T^{a_{t}}z\to y^{(l_{0})}
\end{align*}
as $(t=(j,l_{i}))_{i}$ goes to infinity. The proof is completed by the density of the sequence $(y^{(l)})_{l}$.
\end{proof}

We now make a refinement of the previous proof to obtain that $HC(B_{w})$ contains an algebra that is not finitely generated.

\begin{theorem}\label{genresl}
Let $(X, (\|\cdot\|_q)_q)$ be a Fr\'echet sequence algebra under coordinatewise multiplication in which $(e_{n})_{n}$ is a basis with Property A. Let $B_{w}$ be a hypercyclic weighted backward shift on $X$. If there exists an increasing sequence $(p_k)_k$ of natural numbers satisfying \eqref{eq1} such that
\[
\text{for any $n\geq 0$,}\quad\prod_{\nu=0}^{p_k+n}w_{\nu}^{-1}\to 0 \quad\text{as $k\to\infty$},
\]
then $HC(B_{w})$ contains an algebra, except zero, that is not finitely generated.

In other words, $HC(B_w)$ is algebrable.
\end{theorem}

\begin{proof}
We begin the proof as in Theorem \ref{thrm:algebra1}. With the notation defined there we obtain again a dense sequence $(y^{(l)})_l\subset \varphi$ of non-zero points in $X$ and an increasing sequence of natural numbers $(a_{r})_{r\geq 1}$ with  $a_{r}\in \{p_k: k\geq 1\}$ such that, if $r=(m,l)\in \mathbb N$, then
\begin{enumerate}[label=\textbf{A.\arabic*}]
\item \label{condition1b} $\|	(S^{a_{r}} y^{(l)})^{\frac{1}{m}}\|_r< 2^{-r}$,
\end{enumerate}
and if $r\geq 2$, then
\begin{enumerate}[label=\textbf{A.\arabic*},resume]
\item \label{condition2b} $\| T^{a_{t}}( S^{a_{r}}y^{(l)})^{\frac{\nu}{m}}\|_r< 2^{-r}$ for $1\leq t<r$ and $1\leq\nu \leq d_{r}$,
\item \label{condition3b} $a_{r}-a_{r-1}> s_{\widetilde{l}}$ with $r-1=(\widetilde{m},\widetilde{l})$,
\end{enumerate}
where $d_{r}=\max_{(\widetilde{m},\widetilde{l})< r}\widetilde{m}$.

Let us now consider a partition of the natural numbers into an infinite number of infinite sets $\mathbb N_{k}$, $k\geq 1$, such that the sequence $(y^{(l)})_{l\in\mathbb N_{k}}$ is dense in $X$ for any $k\geq 1$.  We can assume that for each $l_{0}\in\mathbb N_k$ the element $y^{(l_{0})}$ appears infinitely many times in the sequence  $(y^{(l)})_{l\in\mathbb{N}_k}$.

For each natural number $k$ we consider the vector 
\[
x^{(k)}=\sum_{m=1}^{\infty} \sum_{l\in\mathbb N_{k}}	( S^{a_{(m,l)}}y^{(l)})^{\frac{1}{m}}.
\]
It follows from condition \ref{condition1b} that these series converge in $X$, so that $x^{(k)}\in X$. 

Note that, by condition \ref{condition3b} and the fact that the sets $\mathbb{N}_k$ are pairwise disjoint, we have that
\begin{equation}\label{eq:20}
x^{(k)}x^{(k')}=0\text{ if $k\ne k'$.}
\end{equation}

Let $\mathcal{A}$ be the algebra generated by $(x^{(k)})_k$. Since finitely many elements of $\mathcal{A}$ only involve a finite number of the elements $x^{(k)}$, $k\geq 1$, \eqref{eq:20} shows that $\mathcal{A}$ is not finitely generated. Thus, to complete the proof, it suffices to show that any non-zero point in $\mathcal{A}$ is hypercyclic for $T$.

Let $z\in\mathcal{A}\setminus\{0\}$. We can write
\begin{equation}\label{eq:21}
z=\sum_{\substack{\beta\in I \subset\mathbb{N}_0^s\\ \beta\neq 0}}c_{\beta}(x^{(1)})^{\beta_{1}}\cdots (x^{(s)})^{\beta_{s}}
\end{equation}
for some $s\geq 1$ and $I$ finite, where $(x^{(k)})^0=e$. By \eqref{eq:20}, this reduces to
$$
z=\sum_{\nu=j}^{N}Q_{\nu}
$$
with $1\leq j\leq N$ and $Q_j\neq 0$, where $Q_{\nu}$ is the $\nu$-homogeneous part of $z$, 
\[
Q_{\nu}=\sum_{k=1}^{s}c_{\nu,k}(x^{(k)})^{\nu}.
\]
Since $Q_j$ is not zero, we may assume that there is some $k'$ such that $c_{j,k'}=1$.

Let us fix $l_{0}\in\mathbb N_{k'}$. Then, for $t=(j,l_{0})$, a calculation as in \eqref{eq:pointb} together with condition \ref{condition2b} shows that
\begin{align*}
\Vert T^{a_{t}}(Q_{j}) - y^{(l_{0})}\Vert_t
& \leq  \sum_{k=1}^{s} \vert c_{j,k}\vert\sum_{\substack{r\geq t+1\\r=(m,l),\ l\in\mathbb N_{k}}} \Vert T^{a_{t}}(S^{a_{r}}y^{(l)})^{\frac{j}{m}} \Vert_r\\
& \leq  \sum_{k=1}^{s} \vert c_{j,k}\vert \sum_{r=t+1}^{\infty} \Vert  T^{a_{t}}(S^{a_{r}}y^{(l)})^{\frac{j}{m}} \Vert_r\\
&<   2^{-t} \sum_{k=1}^{s} \vert c_{j,k}\vert.
\end{align*}

Since there exists an increasing sequence of natural numbers $(l_{i})_{i}\subset\mathbb N_{k'}$ with $y^{(l_{i})}=y^{(l_{0})}$ for all $i$, we have that
\[
T^{a_{t}}(Q_{j}) \to y^{(l_{0})}
\]
as $(t=(j,l_{i}))_{i}$ goes to infinity.

The same argument as in \eqref{eq:2}, \eqref{eq:4} and \eqref{eq:5} shows that, for any $\nu>j$, $T^{a_{t}}(Q_{\nu})\to 0$ when the sequence $(t=(j,l_{i}))_{i}$ goes to infinity. Hence, 
$$
T^{a_{t}}z\to y^{(l_{0})}
$$ 
when $(t=(j,l_{i}))_{i}$ goes to infinity. The density of the sequence $(y^{(l)})_{l\in\mathbb{N}_{k'}}$ implies that $z\in HC(T)$, which had to be shown.
\end{proof}

The hypothesis on the weight $w$ in Theorems \ref{thrm:algebra1} and \ref{genresl} is slightly technical. However, it allows us to treat general hypercyclic operators in the two cases of greatest interest.  

\begin{corollary}\label{corrolalg}
Let $B_w$ be a hypercyclic weighted backward shift on $\ell^p$, $1\leq p<\infty$, or $c_0$, which we consider as Banach sequence algebras under the coordinatewise multiplication. Then the set $HC(B_w)$ of hypercyclic vectors for $B_w$ is algebrable. This applies, in particular, to the Rolewicz operators $\lambda B$, $|\lambda|>1$.
\end{corollary}

Indeed, since $\|e_n\|=1$ for all $n\geq 0$, the hypothesis on $w$ follows immediately from Proposition \ref{prop:increasingsequence}, that is, from hypercyclicity. Property A follows from part (a) of  Example \ref{propA}.

The space $H(\mathbb{C})$ of entire functions is a Fr\'echet algebra when endowed with the Hadamard product
\[
(f*g)(z) =\sum_{n=0}^\infty a_nb_nz^n,\quad z \in \mathbb C
\]
for $f(z)=\sum_{n=0}^\infty a_nz^n$ and $g(z)=\sum_{n=0}^\infty b_nz^n$, see \cite{RS96}. When we identify entire functions with their sequence of Taylor coefficients at 0 then $H(\mathbb{C})$ turns into a Fr\'echet sequence algebra. Again, since $\|e_n\|_1=1$ for all $n\geq 0$, the hypothesis on $w$ follows immediately from hypercyclicity via Proposition \ref{prop:increasingsequence}.  Property A follows from part (b) of  Example \ref{propA}.

\begin{corollary}\label{corrolalg2}
Let $B_w$ be a hypercyclic weighted backward shift on $H(\mathbb{C})$, which we consider as a Fr\'echet sequence algebra under the Hadamard product. Then the set $HC(B_w)$ of hypercyclic vectors for $B_w$ is algebrable. This applies, in particular, to the MacLane operator $D$ of differentiation.
\end{corollary}

Finally, on the space $\omega$ of all sequences, condition \eqref{eq1} holds trivially for any weighted backward shift, so it no longer implies the hypothesis in the above theorems. We only state here a special case. Recall that  Property A holds for the space $\omega$  by part (c) of  Example \ref{propA}.

\begin{corollary}\label{corrolalg3}
Let $B_w$ be a weighted backward shift on $\omega$, which we consider as a Fr\'echet sequence algebra under coordinatewise multiplication. If  $\prod_{k=0}^n w_k^{-1}\to 0$ as $n\to\infty$, then the set $HC(B_w)$ of hypercyclic vectors for $B_w$ is algebrable. 
\end{corollary}

%%%%%%%%%%%%%%%%%%%%%%%%%%%%%%%%%%%%%%
%%%%%%%%%%%%%%%%%%%%%%%%%%%%%%%%%%%%%%
%%%%%%%%%%%%%%%%%%%%%%%%%%%%%%%%%%%%%%

\section{Fr\'echet sequence algebras under the Cauchy product}
\label{cauchy}

In this section we focus on the study of dynamical systems where the underlying sequence space is a Fr\'echet algebra whose multiplicative structure is given by the Cauchy product. The Cauchy product is the natural structure that appears when we multiply two power series. Given $\sum_{k=0}^\infty a_n z^{n}$ and  $\sum_{n=0}^\infty b_n z^{n}$ we have formally, after regrouping the terms with the same degree, that $(\sum_{n=0}^\infty a_n z^{n}) \cdot (\sum_{n=0}^\infty b_n z^{n}) = \sum_{n=0}^\infty c_n  z^{n}$ where $c_n=\sum_{k=0}^n a_k b_{n-k}$. Even more, if the  power series $\sum_{n=0}^\infty a_n z^{n}$ has radius of convergence $R_{1}$ and the power series  $\sum_{n=0}^\infty b_n z^{n}$ has radius of convergence $R_{2}$, then the resulting power series $\sum_{n=0}^\infty c_n  z^{n}$ has a radius of convergence of at least $\min\{R_{1},R_{2}\}$. In general, the Cauchy product of two sequences $x=(x_n)_n$ and $y=(y_n)_n$ is defined by the discrete convolution 
\begin{equation*}
x \ast y = (z_n)_n,\quad\text{where } z_n=\sum_{k=0}^n x_k y_{n-k}, \quad n\geq 0.
\end{equation*}
For the sake of clarity we will write the Cauchy product of two sequences $x$, $y$ as $x\ast y$, while the $n$-fold Cauchy product will be written as $x^n=x\ast\ldots\ast x$ to avoid a more cumbersome notation.

\begin{example}\label{seqalg}
(a) The most natural example of a Fr\'echet algebra in this scenario is the Fr\'echet space $H(\mathbb{C})$ of entire functions, which we 
consider again as a sequence space via Taylor coefficients at 0, see Example \ref{propA}. Its natural topology is  induced by the family of seminorms
\[
\|(a_n)_{n\geq 0}\|_q = \sup_{|z|\leq q}\Big|\sum_{n=0}^\infty a_n z^n\Big|,\quad q\geq 1.
\]
Then $H(\mathbb{C})$ becomes a Fr\'echet sequence algebra under the Cauchy product.

(b) The sequence space $\ell^1$ is a Banach sequence algebra under its usual norm when endowed with the Cauchy product.

(c) The product topology of the space $\omega$ of all sequences is generated by the increasing sequence of seminorms
$$
\|x\|_q = \sum_{n=0}^q \vert x_{n}\vert,\quad q\geq 1.
$$ 
Then $\omega$ is a Fr\'echet sequence algebra under the Cauchy product.
\end{example}

In order to translate the results obtained in Section \ref{pointwise} to algebras that are defined by Cauchy products we need again to impose conditions on the weight $w$ that defines the weighted backward shift $B_{w}$ and on the basis $(e_n)_n$, as we did in Theorems \ref{thrm:algebra1} and \ref{genresl}. 

As for the weight, we will demand that $B_w$ is mixing. Recall that a weighted backward shift $B_w$ on a Fr\'echet sequence space in which $(e_{n})_{n}$ is a basis is mixing if and only if
\begin{equation}\label{eq:mix}
\frac{1}{\prod_{k=0}^n w_k}e_n\to 0
\end{equation}
in $X$ as $n\to\infty$, see \cite[Theorem 4.8]{GrPe11}. 

As for $(e_n)_n$, we introduce a new property.

\begin{definition}\label{def:propertyB}
Let $(X, (\|\cdot\|_q)_q)$ be a Fr\'echet sequence space that contains the finite sequences. We say that $(e_n)_n$ has \textit{Property B} if
the following conditions hold:
\begin{enumerate}
\item[(i)] there is some $q\geq 1$ such that $\|e_n\|_q>0$ for all $n\geq 0$;
\item[(ii)] for any $r\geq 1$ there is some $q\geq 1$ and some $C_1>0$ such that, for all $n,k\geq 0$,
\[
\|e_n\|_r\cdot \|e_k\|_r\leq C_1 \|e_{n+k}\|_q;
\]
\item[(iii)] for any $m\geq 2$, $M\geq 1$, $r\geq 1$ there is some $\rho\geq 1$ such that for any $t\geq 1$ there is some $\tau\geq 1$ and some $C_2>0$ such that, for any $0\leq k\leq M$, $n\geq M$,
\[
\|e_{mn}\|_t\cdot\|e_{n-k}\|_r\leq C_2\|e_{mn}\|_{\tau}^{\frac{1}{m}}\cdot\|e_{mn-k}\|_{\rho}.
\]
\end{enumerate}
\end{definition}

\begin{lemma}\label{buildingblocks}
Let $(X,(\|\cdot\|_q)_q)$ be a Fr\'echet sequence space in which $(e_{n})_{n}$ is a basis with Property B, and let $B_{w}$ be a mixing weighted backward shift on $X$. Then, for any point
\[
y=\sum_{j=0}^{s}y_je_{j}\in \varphi,
\]
any $m\geq 1$, $r\geq 1$, $N\geq 0$ and $\varepsilon>0$ there are $\eta\geq N$, $\gamma> \eta+2s$, and complex numbers $c_0,\ldots,c_s$ and $b$ such that the point 
\begin{equation*}
p=q+b e_{\gamma}\quad\text{with}\quad
q=\sum_{j=0}^{s}c_{j}e_{\eta+j}
\end{equation*}
satisfies:
\begin{enumerate}[label=\textbf{\emph{C.\arabic*}}]
\item \label{cauchycond2} $\|p\|_r<\varepsilon$;
\item \label{cauchycond3} $mq\ast b^{m-1}e_{(m-1)\gamma}=F_{w^{-1}}^{\eta+(m-1)\gamma}y$;
\item \label{cauchycond4} $\|B_w^{\eta+(m-1)\gamma}(b^{m}e_{m\gamma})\|_r<\varepsilon$.
\end{enumerate}
\end{lemma}

\begin{proof}
For $m=1$ the assertion is trivial. Indeed, take $b=0$ and $c_{j}=\frac{v_jy_j}{v_{\eta+j}}$ for $j=0,\ldots,s$. Then \emph{\ref{cauchycond3}} and  \emph{\ref{cauchycond4}} hold trivially. By \eqref{eq:mix} we may take $\eta\geq N$ so large that 
\[
\|p\|_r=\Big\|\sum_{j=0}^{s}\frac{v_jy_j}{v_{\eta+j}}e_{\eta+j}\Big\|_r<\varepsilon,
\]
hence \emph{\ref{cauchycond2}}. Finally choose any $\gamma> \eta+2s$.

Fix $m$ bigger than one. Let $b\in\mathbb{C}$ be given, where $b\neq 0$. Setting
\begin{equation}\label{eq:alph}
c_{j}=\frac{1}{mb^{m-1}}\frac{v_jy_j}{v_{\eta+j+(m-1)\gamma}}, \quad j=0,\ldots,s,
\end{equation}
we see that condition \emph{\ref{cauchycond3}} holds. Now let  $r\geq 1$, $N\geq 0$ and $\varepsilon\in (0,1]$. It remains to choose $\eta\geq N$, $\gamma>\eta+2s$ and $b$ so that \emph{\ref{cauchycond2}} and \emph{\ref{cauchycond4}} hold. 

 In condition (i) of Property B we may assume that $q=1$, so that $\|e_n\|_r>0$ for all $r\geq 1$ and $n\geq 0$. 

By condition (ii) of Property B, repeated $m-1$ times, there is some $q\geq r$ and some $C_3>0$ such that, for any $n,k\geq 0$,
\begin{equation}\label{eq:iinew}
\|e_n\|_r^{m-1}\cdot \|e_k\|_r\leq C_3 \|e_{(m-1)n+k}\|_q.
\end{equation}
Let
\[
C_4=\frac{1}{m}\sum_{j=0}^s |v_j||y_j|+1\quad\text{and}\quad \widetilde{\varepsilon}=\Big(\frac{\varepsilon}{2C_4}\Big)^2.
\]
In view of \eqref{eq:mix} there is some $M> 2s$ such that
\begin{equation}\label{eq:mix2}
\frac{1}{|v_{n}|}\|e_n\|_q <\min\{1, C_3^{-1}\}\widetilde{\varepsilon}^m
\end{equation}
whenever $n\geq M$. 

Next, let $\rho\geq 1$ be chosen according to condition (iii) of Property B. Since $B_w$ is continuous and
\[
B_w^ke_{m\gamma}= \frac{v_{m\gamma}}{v_{m\gamma-k}}e_{m\gamma-k},\quad \gamma\geq 0, k\leq m\gamma
\]
there is some $t\geq 1$ and $C_5>0$ such that
\[
 \frac{|v_{m\gamma}|}{|v_{m\gamma-k}|}\|e_{m\gamma-k}\|_\rho\leq C_5\|e_{m\gamma}\|_t,\quad k\leq M, \gamma\geq M.
\]

We now take $\tau\geq 1$ as in condition (iii) of Property B, which implies that
\[
\frac{|v_{m\gamma}|}{|v_{m\gamma-k}|}\|e_{\gamma-k}\|_r\leq C_2C_5\|e_{m\gamma}\|_{\tau}^\frac{1}{m},\quad k\leq M, \gamma\geq M.
\]
We deduce that
\[
\frac{|v_{m\gamma}|^{m-1}}{|v_{m\gamma-k}|^m}\|e_{\gamma-k}\|_r^m\leq (C_2C_5)^m\frac{1}{|v_{m\gamma}|}\|e_{m\gamma}\|_{\tau},\quad  k\leq M, \gamma\geq M,
\]
which tends to zero as $\gamma\to \infty$ by \eqref{eq:mix}. Thus there is some $\gamma\geq N+M$ so that with $\eta=\gamma-M$ and $j=0,\ldots,s$ we have that
\[
\frac{|v_{m\gamma}|^{m-1}}{|v_{\eta+j+(m-1)\gamma}|^m}\|e_{\eta+j}\|_r^m\leq 1
\]
and hence, in view of \eqref{eq:mix2} and the fact that $r\leq q$,
\begin{equation}\label{eq:lem1}
\frac{1}{|v_{\gamma-\eta}|^{\frac{1}{m}}}\|e_{\gamma-\eta}\|_r^{\frac{1}{m}}\frac{|v_{m\gamma}|^{\frac{1}{m}}}{|v_{\eta +j+(m-1)\gamma}|^{\frac{1}{m-1}}}\|e_{\eta+j}\|_r^{\frac{1}{m-1}} <\widetilde{\varepsilon}.
\end{equation}
Note that $\eta\geq N$ and $\gamma>\eta+2s$. 

From \eqref{eq:iinew} and \eqref{eq:mix2} we obtain that for these $\gamma$ and $\eta$ and any $j=0,\ldots,s$,
\[
\frac{1}{|v_{\eta+j+(m-1)\gamma}|}\|e_\gamma\|_r^{m-1}\|e_{\eta+j}\|_r \leq \frac{C_3}{|v_{\eta+j+(m-1)\gamma}|}\|e_{\eta+j+(m-1)\gamma}\|_q <\widetilde{\varepsilon}^m,
\]
hence
\begin{equation}\label{eq:lem2}
\frac{1}{|v_{\eta+j+(m-1)\gamma}|^{\frac{1}{m-1}}}\|e_\gamma\|_r\|e_{\eta+j}\|_r^{\frac{1}{m-1}} <\widetilde{\varepsilon}.
\end{equation}

So, let finally
\[
b = \Big(\max_{0\leq j\leq s} \frac{\|e_{\eta+j}\|_r^{\frac{1}{m-1}}}{|v_{\eta+j+(m-1)\gamma}|^{\frac{1}{m-1}}}\cdot
\min\Big\{\frac{1}{\|e_\gamma\|_r },
\frac{|v_{\gamma-\eta}|^{\frac{1}{m}}}{|v_{m\gamma}|^{\frac{1}{m}}\|e_{\gamma-\eta}\|_r^{\frac{1}{m}}}
\Big\}\Big)^{\frac{1}{2}},
\]
which is strictly positive. Then we have with \eqref{eq:alph}, \eqref{eq:lem1} and \eqref{eq:lem2}
\begin{align*}
\|q\|_r &\leq \sum_{j=0}^s |c_j|\|e_{\eta+j}\|_r = \frac{1}{mb^{m-1}}\sum_{j=0}^s |v_j||y_j|\frac{\|e_{\eta+j}\|_r}{|v_{\eta+j+(m-1)\gamma}|}\\
&\leq C_4 \frac{1}{b^{m-1}}\max_{0\leq j\leq s} \frac{\|e_{\eta+j}\|_r}{|v_{\eta+j+(m-1)\gamma}|}\\
&= C_4 \Big(\max_{0\leq j\leq s} \frac{\|e_{\eta+j}\|_r^{\frac{1}{m-1}}}{|v_{\eta+j+(m-1)\gamma}|^{\frac{1}{m-1}}}\cdot
\max\Big\{\|e_\gamma\|_r ,
\frac{|v_{m\gamma}|^{\frac{1}{m}}\|e_{\gamma-\eta}\|_r^{\frac{1}{m}}}{|v_{\gamma-\eta}|^{\frac{1}{m}}}
\Big\}\Big)^{\frac{m-1}{2}}\\
&<C_4\widetilde{\varepsilon}^{\frac{m-1}{2}}\leq \tfrac{\varepsilon}{2}.
\end{align*}
Moreover, \eqref{eq:lem2} implies that
\[
 \|b e_\gamma\|_r =b \|e_\gamma\|_r \leq \max_{0\leq j\leq s}\Big(\frac{1}{|v_{\eta+j+(m-1)\gamma}|^{\frac{1}{m-1}}}\|e_\gamma\|_r\|e_{\eta+j}\|_r^{\frac{1}{m-1}}\Big)^{\frac{1}{2}}<\widetilde{\varepsilon}^{\frac{1}{2}}\leq\tfrac{\varepsilon}{2}.
\]
Altogether we have that
\[
\|p\|_r\leq\|q\|_r+\|b e_\gamma\|_r<\varepsilon,
\]
so that \emph{\ref{cauchycond2}} holds. 

On the other hand, \eqref{eq:lem1} implies that
\begin{align*}
\|B_w^{\eta+(m-1)\gamma}(b^{m}e_{m\gamma})\|_r &=
\Big(b\frac{|v_{m\gamma}|^{\frac{1}{m}}\|e_{\gamma-\eta}\|_r^{\frac{1}{m}}}{|v_{\gamma-\eta}|^{\frac{1}{m}}}\Big)^m\\
&\leq \max_{0\leq j\leq s}\Big(\frac{\|e_{\eta+j}\|_r^{\frac{1}{m-1}}}{|v_{\eta+j+(m-1)\gamma}|^{\frac{1}{m-1}}}\frac{|v_{m\gamma}|^{\frac{1}{m}}\|e_{\gamma-\eta}\|_r^{\frac{1}{m}}}{|v_{\gamma-\eta}|^{\frac{1}{m}}}\Big)^{\frac{m}{2}}<\widetilde{\varepsilon}^{\frac{m}{2}}\leq\varepsilon.
\end{align*}
Thus, \emph{\ref{cauchycond4}} hold as well.
\end{proof}

\begin{remark}\label{remomega}
For the space $\omega$, the sequence $(e_n)_n$ does not have Property B because it does not satisfy condition (i). However, the conclusion of 
Lemma \ref{buildingblocks} holds trivially by choosing $\eta$ and $\gamma$ so large that the value of the seminorms in \emph{\ref{cauchycond2}} and \emph{\ref{cauchycond4}} is zero. Since the remaining results of this section only rely on this conclusion, they also hold for all (not necessarily mixing) weighted backward shifts on $\omega$.
\end{remark}

We can now obtain the analogue of Theorem \ref{thrm:algebra1} for Fr\'echet algebras defined by Cauchy products.
 
\begin{theorem}
\label{thrm:algebra1cauchy}
Let $(X,(\|\cdot\|_q)_q)$ be a Fr\'echet sequence algebra under the Cauchy product in which $(e_{n})_{n}$ is a basis with Property B, and let $B_{w}$ be a mixing weighted backward shift on $X$. Then there exists a point $x\in HC(B_{w})$ such that the algebra generated by $x$, except zero, is contained in $HC(B_{w})$.
\end{theorem}

\begin{proof}
To simplify our notation we will denote, as before, by $T$ the weighted backward shift operator $B_{w}$ on $X$ and by $S$ the weighted forward shift operator $F_{w^{-1}}$. Recall that $TS=I$ on $\omega$.
 
As in the proof of Theorem \ref{thrm:algebra1} we fix a correspondence $r=(m,l)$ between $\mathbb{N}$ and $\mathbb{N}\times\mathbb{N}$, and we fix a dense sequence $(y^{(l)})_l\subset \varphi$ of non-zero points in $X$. 

We define a partition of the set of all non-zero multi-indices by setting 
\begin{equation*}
I_{m,t}=\{\alpha\in\mathbb N_0^{t}:\vert \alpha\vert=m, \alpha_t>0\},\quad m,t\geq 1, 
\end{equation*}
where $|\alpha|=\sum_{j=1}^t \alpha_j$. Given $\alpha\in I_{m,t}$ and $p_1,\ldots,p_t\in X$ we will write 
$$
P^{\alpha}=p_{1}^{\alpha_1}\ast\cdots\ast p_{t}^{\alpha_t};
$$
and 
\[
\binom{m}{\alpha}= \frac{m!}{\alpha_1!\cdots \alpha_t!}
\]
denotes the corresponding multinomial coefficient.

Let us now construct an increasing sequence $(a_{r})_{r\geq 0}$ of natural numbers and a sequence $(p_{r})_{r\geq 0}$ in $\varphi$ satisfying that, if $r=(m,l)\in \mathbb N$, then

\begin{enumerate}[label=\textbf{D.\arabic*}]
\item \label{Cs1} $\Vert p_{r}\Vert_r<2^{-r},$
\item \label{Cs2} $T^{a_{r}}P^{\alpha}=0$  for all $\alpha\in I_{\mu,t}$, $1\leq \mu<m$, $1\leq t\leq r$, or $\mu=m$, $1\leq t< r$, and for all $\alpha\in I_{m,r}$, $\alpha\neq (0,\ldots, 0,m)$,
\item \label{Cs3} $\Vert T^{a_{r}}p_{r}^{m}-y^{(l)}\Vert_{r} < 2^{-r}$,
\item \label{Cs4} $\sum_{\alpha\in I_{\mu, r}} \binom{\mu}{\alpha}\Vert T^{a_{t}}P^{\alpha}\Vert_{r}< 2^{-r}$  for $1\leq t<r$ and $1\leq \mu\leq \widetilde{m}$, where $t=(\widetilde{m},\widetilde{l})$. 
\end{enumerate}
We proceed by induction on $r\geq 0$. For $r=0$ we set $a_0=1$ and $p_0=e_0$; there is nothing else to do. 

Let $r\geq 1$, and assume that we have constructed natural numbers $a_0 < a_{1}<\ldots< a_{r-1}$ and points $p_{0},\ldots, p_{r-1}$ in $\varphi$ satisfying conditions \ref{Cs1}, \ref{Cs2}, \ref{Cs3} and \ref{Cs4}. 

Consider $r=(m,l)$. Let $\varepsilon\leq 2^{-r}$ be a positive number and $\rho\geq r$ an integer, both to be specified later; write
\[
y^{(l)} = \sum_{j=0}^{s_l} y_j^{(l)} e_j.
\]
Let $N$ be the largest index of the non-zero coordinates in any of the points $p_{0},\ldots, p_{r-1}$. 
By Lemma \ref{buildingblocks} there are $\eta > \max\{N,a_{r-1}\}$, $\gamma> \eta+2s_l$ and complex numbers $d_0,\ldots,d_{s_l}$ and $b$ such that the point 
\[
p=q+b e_{\gamma}\quad\text{with}\quad q=\sum_{j=0}^{s_l}d_j e_{\eta+j}
\]
satisfies
\begin{enumerate}[label=\textbf{E.\arabic*}]
\item \label{conda} $\|p\|_\rho<\varepsilon$;
\item \label{condb} $mq\ast b^{m-1}e_{(m-1)\gamma}=S^{\eta+(m-1)\gamma}y^{(l)}$;
\item \label{condc} $\|T^{\eta+(m-1)\gamma}(b^{m}e_{m\gamma})\|_r<2^{-r}$.
\end{enumerate}
We define
\[
p_r=p, \quad a_r=\eta+(m-1)\gamma.
\]
Then $a_r\geq \eta > a_{r-1}$. Moreover, \ref{conda} implies condition \ref{Cs1} since $\rho\geq r$ and $\varepsilon\leq 2^{-r}$. 

Let $\alpha\in I_{\mu,t}$, $1\leq \mu\leq m$, $1\leq t\leq r$. If $\mu<m$ then the largest index of the non-zero coordinates of $P^\alpha$ is at most $(m-1)
\gamma<a_r$ (note that $N\leq \gamma$); now, if $\mu=m$ and  $t<r$ then this index as at most $mN=N+(m-1)N< \eta+(m-1)\gamma=a_r$; if $\mu=m$, $t=r$ and $\alpha\neq (0,\ldots,0,m)$ then this index is at most $N+(m-1)\gamma< \eta+(m-1)\gamma=a_r$. Thus, in any case, we have that $T^{a_r}P^\alpha=0$, hence \ref{Cs2}. 

Next, we have that
\begin{align*}
p_{r}^{m} &=\sum_{k=0}^{m}\tbinom{m}{k}q^{m-k}\ast(b e_\gamma)^{k}\\
&=\sum_{k=0}^{m-2}\tbinom{m}{k}q^{m-k}\ast(b e_\gamma)^{k} + mq\ast b^{m-1} e_{(m-1)\gamma} + b^me_{m\gamma}.
\end{align*}
The largest index of the non-zero coordinates of the first sum is at most
\[
2(\eta+s_l)+(m-2)\gamma = \eta+(\eta+2s_l) +(m-2)\gamma < \eta +(m-1)\gamma=a_r,
\]
so that $T^{a_r}$ sends the sum to 0. Hence 
\[
T^{a_r} p_r^m =  T^{a_r}(mq\ast b^{m-1} e_{(m-1)\gamma}) + T^{a_r}(b^me_{m\gamma})= y^{(l)}+ T^{a_r}(b^me_{m\gamma}),
\]
where we have applied \ref{condb} and the fact that $TS=I$. Thus, \ref{condc} implies condition \ref{Cs3}.

Finally, condition \ref{Cs4} consists of a finite number of inequalities (in fact, for $r=1$ the condition is empty). Now, if $\alpha\in I_{\mu,r}$, %$1\leq t\leq r-1$, $1\leq \mu\leq \widetilde{m}$, where $t=(\widetilde{m}, \widetilde{l})$, 
then $P^\alpha$ is of the form
\[
p_1^{\alpha_1}\ast\cdots\ast p_r^{\alpha_r}
\]
with $\alpha_r\neq 0$. Since $p_1,\ldots, p_{r-1}$ are known and both the Cauchy product and the operator $T$ are continuous on $X$, there exist $\rho\geq r$ and $\varepsilon\leq 2^{-r}$ such that all the inequalities in \ref{Cs4} are satisfied as soon as $\|p_r\|_\rho<\varepsilon$. We choose $\rho$ and $\varepsilon$ so that these inequalities hold.

This completes the induction process. Consider now
\begin{equation*}
x=\sum_{r=1}^{\infty}p_{r}.
\end{equation*}
As a consequence of \ref{Cs1}, the series converges and $x\in X$. We claim that the algebra generated by $x$ is contained in $HC(T)$, except for zero. Thus let
\[
z=\sum_{\mu=1}^{m}c_{\mu}x^{\mu}
\]
with $c_{1},\ldots,c_{m}\in \mathbb C$ and $c_{m}\ne 0$. We may assume that $c_{m}=1$.

Let $l\geq 1$. Since 
\[
x^m = \sum_{t=1}^\infty \sum_{\alpha\in I_{m,t}}\tbinom{m}{\alpha} P^\alpha
\]
we have for $r=(m,l)$ that, in view of condition \ref{Cs2},
\begin{align*}
T^{a_{r}}x^{m}-y^{(l)}
&= \sum_{t<r} \sum_{\alpha\in I_{m,t}}\tbinom{m}{\alpha} T^{a_{r}}P^\alpha + \sum_{\substack{\alpha\in I_{m,r}\\ \alpha\neq (0,\ldots,0,m)}}\tbinom{m}{\alpha} T^{a_{r}}P^\alpha\\
&\phantom{xxxxxxxxxxxxx} + T^{a_{r}}p_{r}^{m}-y^{(l)} + \sum_{t>r} \sum_{\alpha\in I_{m,t}}\tbinom{m}{\alpha} T^{a_{r}}P^\alpha\\
&= T^{a_{r}}p_{r}^{m}-y^{(l)} + \sum_{t>r} \sum_{\alpha\in I_{m,t}}\tbinom{m}{\alpha} T^{a_{r}}P^\alpha, 
\end{align*}
and hence
\begin{equation*}\label{degreem}
\begin{split}\|T^{a_{r}}x^{m}-y^{(l)}\|_r &< 2^{-r}+\sum_{t>r} \sum_{\alpha\in I_{m,t}}\tbinom{m}{\alpha}\| T^{a_{r}}P^\alpha\|_r\quad\text{(by condition \ref{Cs3})}\\
 &< 2^{-r}+ \sum_{t>r}2^{-t}\quad\text{(by condition \ref{Cs4})}\\
 &= 2^{-r+1}.
\end{split}
\end{equation*}
In the same way we obtain by conditions \ref{Cs2} and \ref{Cs4} for $\mu<m$
\begin{equation*}
\label{goestozero}
\|T^{a_{r}}x^{\mu}\|_r \leq\sum_{t>r} \sum_{\alpha\in I_{\mu,t}}\tbinom{\mu}{\alpha}\|T^{a_{r}}P^\alpha\|_r< 2^{-r}.
 \end{equation*}
Altogether we have that
\[
\|T^{a_{r}}z-y^{(l)}\|_r < \sum_{\mu=1}^{m-1}|c_\mu|2^{-r} + 2^{-r+1} = \Big(\sum_{\mu=1}^{m-1}|c_\mu|+2\Big)2^{-r}.
\]
By the density of the sequence $(y^{(l)})_{l}$ the result follows and the proof is complete.
\end{proof}

We next want to show that the set $HC(B_w)$ is even algebrable. Thus we need to pass from an algebra generated by a single point $x$ to one generated by infinitely many points $x^{(k)}$, $k\geq 1$. The building blocks will be essentially the same points $p_r$ as in the previous proof. However, we need to ensure that the algebra generated by the $x^{(k)}$ is not finitely generated. This can be achieved by choosing suitable coefficients for the $p_r$.

\begin{theorem}
\label{thrm:algebraseveralcauchy}
Let $(X,(\|\cdot\|_q)_q)$ be a Fr\'echet sequence algebra under the Cauchy product in which $(e_{n})_{n}$ is a basis with Property B, and let $B_{w}$ be a mixing weighted backward shift on $X$. Then $HC(B_{w})$ contains an algebra, except zero, that is not finitely generated.
In other words, $HC(B_w)$ is algebrable.
\end{theorem}

\begin{proof} The proof of Theorem \ref{thrm:algebra1cauchy} will be modified in certain ways. We write again $T=B_w$, and we let $(y^{(l)})_l\subset \varphi$ be a dense sequence of non-zero points in $X$. 

We will here identify $\mathbb{N}$ with $\mathbb{N}^3$, so that we write $r=(m,l,\nu)\in \mathbb{N}$ with $m,l,\nu\geq 1$. Moreover, let $A$ be a countable dense subset of the set of finite sequences of norm at most 1 in $\ell^\infty(\mathbb{N})$. Let $\Lambda=(\lambda_{k,\nu})_{k,\nu\geq 1}$ be a matrix so that each column belongs to $A$, and each element of $A$ appears infinitely often as a column. 

Let $I_{m,t}$, $m,t\geq 1$, and $P^{\alpha}$, $\alpha\in I_{m,t}$, be defined as in the proof of Theorem \ref{thrm:algebra1cauchy}. Following that proof we can then construct an increasing sequence $(a_{r})_{r\geq 1}$ of natural numbers and a sequence $(p_r)_{r\geq 1}$ in $\varphi$ such that, if $r=(m,l,\nu)\in \mathbb N$, then
\begin{enumerate}[label=\textbf{F.\arabic*}]
\item \label{Ds1} $\Vert p_{r}\Vert_r<2^{-r},$
\item \label{Ds2} $T^{a_{r}}P^{\alpha}=0$  for all $\alpha\in I_{\mu,t}$, $1\leq \mu<m$, $1\leq t\leq r$, or $\mu=m$, $1\leq t< r$, and all $\alpha\in I_{m,r}$, $\alpha\neq (0,\ldots, 0,m)$,
\item \label{Ds3} $\Vert T^{a_{r}}p_{r}^{m}-y^{(l)}\Vert_{r} < 2^{-r}$,
\item \label{Ds4} $\Vert T^{a_{t}}P^{\alpha}\Vert_r< 2^{-r}$  for all $\alpha\in I_{\mu,r}$ with $1\leq t<r$ and $1\leq\mu\leq \widetilde{m}$, where $t=(\widetilde{m},\widetilde{l},\widetilde{\nu})$.  
\end{enumerate}
We may achieve, in addition, that if $p_r=\sum_{j=\eta_r}^{\gamma_r}d_je_j$ with $d_{\gamma_r}\neq 0$ then $a_{r}\leq m\gamma_{r}<\eta_{r+1}$, where $r=(m,l,\nu)$.

We now define, for any $k\geq 1$, 
\begin{equation}\label{eq:avoid0}
x^{(k)} = \sum_{r=1}^{\infty}\lambda_{k,\nu_r} p_{r}.
\end{equation}
Since the elements of the matrix $\Lambda$ are bounded (by 1), \ref{Ds1} implies that these series converge, so that $x^{(k)}\in X$, $k\geq 1$. 

Let $\mathcal{A}$ be the algebra generated by the points $x^{(k)}$, $k\geq 1$. We first show that any non-zero point $z\in\mathcal{A}$ is hypercyclic for $T$. We can write
\[
z=\sum_{\substack{\beta\in I \subset \mathbb{N}_0^s\\ \beta\neq 0}}c_{\beta}(x^{(1)})^{\beta_{1}}\ast\cdots\ast (x^{(s)})^{\beta_{s}}
\]
for some $s\geq 1$ and $I$ finite. Let
\[
m= \max\{|\beta| : c_\beta\neq 0\}.
\]
Thus
\[
z= \sum_{\mu=1}^m \sum_{|\beta|=\mu}c_{\beta}(x^{(1)})^{\beta_{1}}\ast\cdots\ast (x^{(s)})^{\beta_{s}}.
\]

One reason for introducing the $\lambda_{k,\nu}$ in \eqref{eq:avoid0} is that one cannot be sure that $\sum_{|\beta|=m}c_{\beta}\neq 0$. But since the polynomial $P(a_1,\ldots,a_s)= \sum_{|\beta|=m}c_{\beta}a_1^{\beta_{1}}\cdots a_s^{\beta_{s}}$ is non-zero and since the first $s$ coordinates of the elements of $A$ are dense in the polydisk of $\mathbb{C}^s$, there is an element $a=(a_n)_n\in A$ such that 
\[
\sum_{|\beta|=m}c_{\beta}a_1^{\beta_{1}}\cdots a_s^{\beta_{s}}=:\rho\neq 0.
\]

Now, in order to show that $z$ is hypercyclic, let $l\geq 1$. By the definition of the matrix $\Lambda$ there is some $\nu\geq 1$, arbitrarily large, such that
\[
\lambda_{k,\nu}=a_k,\quad k=1,\ldots,s.
\]
Let $r=(m,l,\nu)$, which can be made arbitrarily large by choosing $\nu$ large. 

After expansion, taking account of the continuity of the Cauchy product, we see that there are complex numbers $d_{\alpha}$, $\alpha\in I_{m,t}$, $t\geq 1$, such that
\begin{equation}\label{eq:exp}
\begin{split}
\sum_{|\beta|=m}c_{\beta}(x^{(1)})^{\beta_{1}}\ast\cdots\ast (x^{(s)})^{\beta_{s}}&=\sum_{t=1}^{r-1}\sum_{\alpha\in I_{m,t}} d_\alpha P^\alpha + \sum_{\substack{\alpha\in I_{m,r}\\\alpha\neq (0,\ldots,0,m)}} d_{\alpha} P^{\alpha}\\ &\phantom{xxxxxxxxxxxxxx}+ \rho p_r^m + \sum_{t>r}\sum_{\alpha\in I_{m,t}} d_{\alpha} P^{\alpha};
\end{split}
\end{equation}
note that the coefficient of $p_r^m$ is
\[
d_{(0,\ldots,0,m)}=\sum_{|\beta|=m}c_{\beta}\lambda_{1,\nu_r}^{\beta_{1}}\cdots \lambda_{s,\nu_r}^{\beta_{s}}=\sum_{|\beta|=m}c_{\beta}a_1^{\beta_{1}}\cdots a_s^{\beta_{s}} = \rho
\]
since $\nu_r=\nu$.

Let 
\[
C_\mu= (1+\mu)^s\max_{|\beta|=\mu}|c_\beta|,\quad 1\leq \mu\leq m.
\]
Now, each $P^\alpha$ comes from one of the $t^m$ terms without any power of $p_{t+1}, p_{t+2},\ldots$ in the expansion of $(x^{(1)})^{\beta_{1}}\ast\cdots\ast (x^{(s)})^{\beta_{s}}$, and there are at most $(1+m)^s$ choices of $\beta$ with $|\beta|=m$; moreover, the elements of $\Lambda$ are bounded by 1. Altogether we obtain as a very rough estimate that
\begin{equation}\label{eq:delta}
|d_{\alpha}|\leq C_{m} t^m,\quad \alpha\in I_{m,t}, t\geq 1.
\end{equation}

It follows from \eqref{eq:exp} with condition \ref{Ds2} that
\[
T^{a_r}\Big(\sum_{|\beta|=m}c_{\beta}(x^{(1)})^{\beta_{1}}\ast\cdots\ast (x^{(s)})^{\beta_{s}}\Big)  = \rho T^{a_r}p_r^m + \sum_{t>r}\sum_{\alpha\in I_{m,t}} d_{\alpha} T^{a_r} P^{\alpha},
\] 
and therefore, by conditions \ref{Ds3} and \ref{Ds4} with \eqref{eq:delta},
\[
\Big\|T^{a_r}\Big(\sum_{|\beta|=m}c_{\beta}(x^{(1)})^{\beta_{1}}\ast\cdots\ast (x^{(s)})^{\beta_{s}}\Big) - \rho  y^{(l)}\Big\|_r < \rho 2^{-r} + \sum_{t>r}\text{card} (I_{m,t})C_m t^m 2^{-t}.
\] 

In the same way, for $1\leq \mu <m$, there are complex numbers $d_{\alpha}$, $\alpha\in I_{\mu,t}$, $t\geq 1$, such that
\begin{equation}\label{eq:exp2}
\begin{split}
\sum_{|\beta|=\mu}c_{\beta}(x^{(1)})^{\beta_{1}}\ast\cdots\ast (x^{(s)})^{\beta_{s}}&=\sum_{t=1}^{r}\sum_{\alpha\in I_{\mu,t}} d_\alpha P^\alpha  + \sum_{t>r}\sum_{\alpha\in I_{\mu,t}} d_{\alpha} P^{\alpha}.
\end{split}
\end{equation}
From \ref{Ds2}, \ref{Ds4} we thus obtain that
\[
\Big\|T^{a_r}\Big(\sum_{|\beta|=\mu}c_{\beta}(x^{(1)})^{\beta_{1}}\ast\cdots\ast (x^{(s)})^{\beta_{s}}\Big)\Big\|_r < \sum_{t>r}\text{card} (I_{\mu,t})C_{\mu} t^{\mu} 2^{-t}.
\] 
Altogether we have that
\[
\|T^{a_r}z-\rho y^{(l)}\|_r < \rho 2^{-r}+\sum_{\mu=1}^m\sum_{t>r}\text{card} (I_{\mu,t})C_{\mu} t^{\mu} 2^{-t}.
\]
Since 
\[
\text{card}(I_{\mu,t}) \leq \tbinom{\mu+t-1}{\mu}\leq \tfrac{(\mu+t)^\mu}{\mu!}
\]
the above series converge. Thus, for any $N\geq 1$ and $\varepsilon>0$ we can find an $r\geq N$ such that
\[
\|T^{a_r}z-\rho y^{(l)}\|_N <\varepsilon.
\]
Since the sequence $(\rho y^{(l)})_l$ is dense in $X$ we deduce that $z$ is hypercyclic for $T$.

The choice of the $\lambda_{k,\nu}$ also ensures that $\mathcal{A}$ is not finitely generated. Indeed, if it were finitely generated, we would have that, for some $s\geq 1$,
\[
x^{(s+1)}=\sum_{\substack{\beta\in I \subset \mathbb{N}_0^s\\ \beta\neq 0}}c_{\beta}(x^{(1)})^{\beta_{1}}\ast\cdots\ast (x^{(s)})^{\beta_{s}}
\]
with complex numbers $c_\beta$, where $I$ is a finite set. Let again $m= \max\{|\beta| : c_\beta\neq 0\}$. As above we can then find some $\nu\geq 1$ such that
\[
\sum_{|\beta|=m}c_{\beta}\lambda_{1,\nu}^{\beta_{1}}\cdots \lambda_{s,\nu}^{\beta_{s}}=:\rho\neq 0.
\]
In view of \eqref{eq:exp} and \eqref{eq:exp2}, we can then write
\[
x^{(s+1)}=\sum_{\substack{1\leq\mu\leq m\\1\leq t\leq r\\(\mu,t)\neq (m,r)}}\sum_{\alpha\in I_{\mu,t}} d_\alpha P^\alpha + \sum_{\substack{\alpha\in I_{m,r}\\\alpha\neq (0,\ldots,0,m)}} d_{\alpha} P^{\alpha}+ \rho p_r^m + \sum_{\mu=1}^m\sum_{t>r}\sum_{\alpha\in I_{\mu,t}} d_{\alpha} P^{\alpha},
\]
where $r=(m,l,\nu)$; note that $l\geq 1$ can be chosen freely. On the right-hand side, the first two terms represent a sequence whose non-zero coordinates have index less than $a_r$ by \eqref{Ds2}, while the non-zero coordinates of the fourth term have index at least $\eta_{r+1}$. Since $a_r\leq m\gamma_r<\eta_{r+1}$, it follows that 
\[
x^{(s+1)}_{m\gamma_r}\neq 0.
\]
However, for the same reason and by the definition of $x^{(s+1)}$, we have that $x^{(s+1)}_{m\gamma_r} =0$ whenever $m\geq 2$.

Thus we must have that $m=1$. But then there are complex numbers $c_k$ such that
\[
x^{(s+1)} = \sum_{k=1}^s c_k x^{(k)} =\sum_{r=1}^\infty \Big(\sum_{k=1}^s c_k \lambda_{k,\nu_r}\Big) p_r.
\]
Now, by the choice of the matrix $\Lambda$ there is some $\nu\geq 1$ such that $\sum_{k=1}^s c_k \lambda_{k,\nu}-\lambda_{s+1,\nu}\neq 0$. This contradicts the fact that $x^{(s+1)}=\sum_{r=1}^\infty \lambda_{s+1,\nu_r} p_r$. 

Therefore $\mathcal{A}$ cannot be finitely generated.
\end{proof}

We spell out the two cases of greatest interest. In each case property B is verified, see Example \ref{seqalg}; note that $\|e_n\|_q=q^n$ in $H(\mathbb{C})$.

\begin{corollary}\label{corrolalgCP1}
Let $B_w$ be a mixing weighted backward shift on $\ell^1$, which we consider as a Banach sequence algebra under the Cauchy product. Then the set $HC(B_w)$ of hypercyclic vectors for $B_w$ is algebrable. This applies, in particular, to the Rolewicz operators $\lambda B$, $|\lambda|>1$.
\end{corollary}

\begin{corollary}\label{corrolalgCP2}
Let $B_w$ be a mixing weighted backward shift on $H(\mathbb{C})$, which we consider as a Fr\'echet sequence algebra under the pointwise product of functions. Then the set $HC(B_w)$ of hypercyclic vectors for $B_w$ is algebrable. This applies, in particular, to the MacLane operator $D$ of differentiation.
\end{corollary}

Finally, any weighted backward shift $B_w$ on $\omega$ satisfies \eqref{eq:mix}. Thus, in view of Remark \ref{remomega}, we have the following.

\begin{corollary}\label{corrolalgCP3}
For any weighted backward shift $B_w$ on $\omega$, considered as a Fr\'echet sequence algebra under the Cauchy product, the set $HC(B_w)$ of hypercyclic vectors for $B_w$ is algebrable. 
\end{corollary}

\end{document}